\documentclass[11pt]{amsart}

\usepackage{amssymb}
\usepackage{mathtools}
\usepackage{tikz,ifthen}
\numberwithin{equation}{section}
\DeclarePairedDelimiter\abs{\lvert}{\rvert}%
\DeclarePairedDelimiter\norm{\lVert}{\rVert}%

\makeatletter
\let\oldabs\abs
\def\abs{\@ifstar{\oldabs}{\oldabs*}}
\let\oldnorm\norm
\def\norm{\@ifstar{\oldnorm}{\oldnorm*}}
\makeatother

\theoremstyle{plain}
\newtheorem{thm}{Theorem}[section]
\newtheorem{theorem}{Theorem}
\newtheorem{lemma}[theorem]{Lemma}

\newtheorem{prop}[thm]{Proposition}
\newtheorem{cor}[thm]{Corollary}
\theoremstyle{definition}

\newtheorem{conj}{Conjecture}

\theoremstyle{remark}

\newcommand\R{\mathbb{R}}

\newcommand\N{\mathbb{N}}

\newcommand{\zn}[1]{\mathbb{Z}^{#1}}

\newcommand{\inn}[1]{\langle #1 \rangle}
\newcommand{\I}{J^\gamma_\lambda}
\newcommand{\IV}{\mathcal{J}^\gamma_\lambda}
\newcommand{\II}{J^\gamma_{\lambda,j}}
\newcommand{\IK}{I_{k,\lambda}}
\newcommand{\JK}{J_{k,\lambda}}
\newcommand{\JaK}{J^a_{\lambda}}

\newcommand{\sumab}[2]{\sum_{\substack{ {#1} \\ {#2} }}}

\newcommand{\A}{N^\gamma_{r}(P;h)}
\newcommand{\AN}{\mathcal{N}^\gamma_{r}(P;h)}
\newcommand{\AL}[1]{N^\gamma_{#1}(P;h)}
\newcommand{\AZ}[1]{N^\gamma_{#1}(P;0)}

\newcommand{\JaC}{\mathbf{J}^a_{s}(P)}

\newcommand{\JakC}{\mathbf{J}^{a_k}_{s}(P)}

\newcommand{\MV}{\mathbf{J}_{s,k}(P)}
\newcommand{\lse}{\ll}

\newcommand{\Z}{\mathbb{Z}}

\begin{document}
\title[Discrete fractional integrals]{On discrete fractional integral operators and related Diophantine equations}
\author{Jongchon Kim}
\thanks{Supported in part by Study Abroad Scholarship Program by NIIED and NSF grant 1201314.} 
\address{Department of Mathematics, University of Wisconsin-Madison, Madison, WI 53706 USA}
\email{jkim@math.wisc.edu}
\begin{abstract}
We study discrete versions of fractional integral operators along curves and surfaces. $l^p \to l^q$ estimates are obtained from upper bounds of the number of solutions of associated Diophantine systems. In particular, this relates the discrete fractional integral along the curve $\gamma(m) = (m,m^2,\cdots,m^k)$ to Vinogradov's mean value theorem. Sharp $l^p \to l^q$ estimates of the discrete fractional integral along the hyperbolic paraboloid in $\Z^3$ are also obtained except for endpoints.
\end{abstract}
\maketitle
\section{Introduction} \label{sec:intro}
The classical Hardy-Littlewood-Sobolev inequality gives the $L^p(\R^d) \to L^q(\R^d)$ bounds for the fractional integral operator $f \to |\cdot|^{-d\lambda} * f$ for $0<\lambda <1$, $1<p<q<\infty$, and $\frac{1}{q} = \frac{1}{p} - (1-\lambda)$. The $l^p(\Z^d) \to l^q(\Z^d)$ bounds for its discrete analogue $f \to \sum_{m\in \Z^d\setminus{0}} |m|^{-d\lambda} f(\cdot - m)$ follows from the $L^p(\R^d) \to L^q(\R^d)$ bounds by a simple comparison argument \cite{SW}. However, if one considers discrete analogues of recent variants of fractional integrals, where the integration is taken over a sub-manifold, the argument fails and the problem becomes more interesting.

Let us give a few examples. Consider the operators $\IK$ and $\JK$ defined by
\begin{equation*}
\IK(f)(n) = \sum_{m=1}^\infty \frac{f(n-m^k)}{m^{\lambda}},
\end{equation*}
\begin{equation*}
\JK(f)(n_1,\cdots,n_k) = \sum_{m=1}^\infty \frac{f(n_1-m,n_2-m^2,\cdots,n_k-m^k)}{m^{\lambda}}.
\end{equation*}
The study of $l^p \to l^q$ estimates of the operator $\IK$ and $\JK$ was initiated by Stein and Wainger \cite{SW}, where they obtained almost sharp bounds for  $k=2$ and $1/2< \lambda<1$ by employing the circle method. Oberlin \cite{O} obtained sharp results for $k=2$ and $0<\lambda<1$ except for endpoints without using the circle method. The endpoint bounds were fully established in a series of papers by Stein and Wainger \cite{SW, SW2}, and Ionescu and Wainger \cite{IW}, but only in the case $k=2$. The case $k\geq3$ seems to be substantially harder. Pierce \cite{P2} studied the operator $\IK$ for $k\geq3$ using tools from number theory, but it is open for $k\geq 3$ in the full $p,q$ range. See \cite{P1} for a generalization of $J_{2,\lambda}$ to quadratic forms in $\Z^d$.

In this paper, we shall consider the operator $\I$ acting on (initially) compactly supported functions $f:\zn{d} \to \mathbb{C}$ by
\begin{equation*}
\I(f)(n) =  \sum_{m=1}^\infty \frac{f(n- \gamma(m))}{m^{\lambda}},
\end{equation*}
where $\gamma : \mathbb{N} \to \zn{d}$ is an injection. Note that $\IK$ and $\JK$ are special cases of $\I$. 

It is known \cite[Proposition (b)]{SW} that the operator $\I$ extends to a bounded operator from $l^p(\zn{d})$ to $l^q(\zn{d})$ if $\frac{1}{q} \leq \frac{1}{p} - (1-\lambda)$ and $1<p<q<\infty$ for any injection $\gamma$. This result is sharp if $\gamma(m)=m$, but not in general. We are interested in obtaining a sharper estimate which is sensitive to $\gamma$.

The previous work described above reveals a close connection between $l^p \to l^q$ bounds and the number of solutions of Diophantine equations related to the curves at least in two different ways. One approach \cite{SW,SW2,P2} is via the Fourier multipliers associated with the operators and Parseval's identity. Another approach \cite{O} is via the combinatorial argument by Christ \cite{C} which originated from the study of an averaging operator along a curve.

The Diophantine equation considered in \cite{P2} for the case $\gamma(m)=m^k$ was 
\begin{equation} \label{eqn:Weyl}
x_1^{k} + \cdots + x_s^{k} = x_{s+1}^{k} + \cdots + x_{2s}^{k}
\end{equation}
for $x_i  \in [1,P]\cap \mathbb{Z}$, $s,P \in \N$. The number of solutions of (\ref{eqn:Weyl}) is known as the mean values of Weyl sums, and has applications to Waring's problem. 

Following the combinatorial approach \cite{O}, for general $\gamma$, we shall relate $l^p \to l^q$ bounds of $\I$ to the Diophantine system 
\begin{equation*} 
\gamma(x_1) + \cdots + \gamma(x_s) = \gamma(x_{s+1}) + \cdots + \gamma(x_{2s}).
\end{equation*}
Moreover, we shall show that it is desirable to obtain estimates on the number of solutions of  Diophantine systems with odd-number of unknowns, which turns out to extend the allowable $\lambda$ range. See Section \ref{sec:Waring} and \ref{sec:hyp}.

Let us introduce a notation. For fixed $h \in \zn{d}$, $r,P \in \mathbb{N}$, let $\A$ be the number of $r$-tuples $(x_1,\cdots,x_r)$ of positive integers $x_i \leq P$ satisfying the equation
\begin{equation} \label{eqn:Deq}
\sum_{i=1}^r (-1)^{i+1} \gamma(x_i) = h.
\end{equation}
Note that there is a trivial estimate $\A \ll P^r$, where $\ll$ denotes the Vinogradov's notation. The following theorem provides $l^p \to l^q$ estimates for $\I$ given a non-trivial estimate on $\A$. In what follows, we allow implied constants to depend on $\gamma$, $r$, and $\epsilon$. 
Throughout the paper, we assume $1\leq p,q \leq \infty$. 
\begin{thm} \label{thm:var}
Suppose that for fixed $r\in\mathbb{N}$ and $\delta>0$, we have $\A \lse P^{r-\delta+\epsilon}$ for each $\epsilon>0$ uniformly in $h \in \zn{d}$. Let $s \in\mathbb{N}$ be the number such that we have either $r=2s$ or $r=2s-1$. Then $\I$ extends to a bounded operator from $l^p(\zn{d})$ to $l^q(\zn{d})$ if $1-\frac{\delta}{r}<\lambda < 1$ and $p,q$ satisfy 
\begin{itemize}
\item[(i)] $\frac{1}{q} < \frac{1}{p} - \frac{1-\lambda}{\delta}$ and
\item[(ii)] $\frac{1}{p}> \frac{s}{\delta}(1-\lambda)$, $\frac{1}{q} < 1 - \frac{s}{\delta}(1-\lambda)$.
\end{itemize}
\end{thm} 

Several remarks are in order.
\begin{enumerate}
\item When $r=2s$, it is enough to assume that $\AZ{2s} \lse P^{2s-\delta+\epsilon}$ since $\AL{2s} \leq \AZ{2s}$ for all $h \in \zn{d}$. This can be seen by writing
\begin{equation*}
\AL{2s} = \int_{[0,1]^d} |S^\gamma( \alpha )|^{2s} e(-h\cdot \alpha) d \alpha, 
\end{equation*}
where $e(t)\equiv e^{2\pi i t}$ and $S^\gamma(\alpha) = \sum_{m=1}^{P} e(\gamma(m)\cdot \alpha)$.

\item Given the stronger estimate $\A \lse P^{r-\delta}$, it is possible to replace $<$ in (i) by $\leq$ with a slight modification of the proof of Theorem \ref{thm:var}. This can be done by applying an abstract analogue (\cite[Section 6.2]{CSWW}) of an interpolation argument of Bourgain \cite{B}, but we shall not pursue it here.

\nocite{IW} 
\item If one considers operators with summation over $m\in\Z\setminus 0$, then the condition $x_i \in [1,P]\cap \Z$ in (\ref{eqn:Deq}) changes to $x_i \in [-P,P]\cap \Z$. See Section \ref{sec:hyp} for an analogous statement in higher dimensions.
\end{enumerate}

Before we turn to the proof of Theorem \ref{thm:var}, we shall give applications for the case $\gamma^a(m)=(m^{a_1},m^{a_2}, \cdots,m^{a_d})$ by using Vinogradov's mean value theorem in Section \ref{sec:jak}. We prove Theorem \ref{thm:var} and a sharp (up to endpoints) $l^p \to l^q$ bound for the discrete fractional integral along the hyperbolic paraboloid in $\zn{3}$ in Section \ref{sec:pf} and Section \ref{sec:hyp}, respectively. In Appendix \ref{sec:alt}, we generalize the Fourier multiplier approach \cite{P2} for operators $\JaK$ considered in Section \ref{sec:jak}, giving an alternative proof of Theorem \ref{thm:JaK}.

\section{The operator $\JaK$} \label{sec:jak}
We study the operator $\JaK \equiv J^{\gamma^a}_\lambda$, where $\gamma^a(m) = (m^{a_1}, \cdots, m^{a_d})$ for a d-tuple of strictly increasing natural numbers $a=(a_1,\cdots,a_d)$.
We define $\norm{a} = a_1 + \cdots + a_d$.
\begin{conj} \label{conj:JK}
Let $0<\lambda< 1$. $\JaK$ extends to a bounded operator from $l^p(\zn{d})$ to $l^q(\zn{d})$ if and only if $p$ and $q$ satisfy
\begin{itemize}
\item[(i)] $\frac{1}{q}  \leq  \frac{1}{p} - \frac{1}{\norm{a}}(1-\lambda)$ and
\item[(ii)] $\frac{1}{p} > 1-\lambda, \frac{1}{q}  <  \lambda$.
\end{itemize}
\end{conj}
See Appendix \ref{sec:app3} for the necessity of conditions (i) and (ii). One is mainly interested in proving Conjecture \ref{conj:JK} for the range of $\frac{\norm{a}-1}{2 \norm{a}-1} < \lambda <1$, since then one may get the full result by interpolating the result with the trivial $l^1(\zn{d}) \to l^\infty(\zn{d})$ bound for $\Re(\lambda) \geq 0$.

In view of Theorem \ref{thm:var}, the operator $\JaK$ is related to the quantity $N^{\gamma^a}_{r}(P;h)$. For even $r=2s$, let us denote $N^{\gamma^a}_{2s}(P;0)$ by $\JaC$, i.e. the number of solutions of the Diophantine system
\begin{equation*} 
x_1^{a_j} + \cdots + x_s^{a_j} = x_{s+1}^{a_j} + \cdots + x_{2s}^{a_j}
\end{equation*}
for $1\leq j \leq d$ with $x_i  \in [1,P]\cap \mathbb{Z}$. By a standard argument (see \cite{KF} or \cite{V}), one may show that there is a lower bound 
\begin{equation} \label{eqn:lower}
P^s+P^{2s-\norm{a}} \ll \JaC.
\end{equation}

It is natural to ask if $P^s+P^{2s-\norm{a}}$ is the true order of $\JaC$. This question for the case $a=a_k\equiv (1,2,\cdots,k)$ has been of great interest. Non-trivial upper bounds for $\MV\equiv \JakC$ are collectively known as Vinogradov's mean value theorem. Recent work by Wooley \cite{W,W2} and Ford and Wooley \cite{FW} report the following substantial progress on Vinogradov's mean value theorem.
\begin{theorem} \label{thm:wooley}
Suppose that $s$ and $k\geq 3$ are natural numbers and that $1\leq s \leq \frac{(k+1)^2}{4}$ or $s\geq k^2-1$. Then for each $\epsilon>0$, one has
\begin{equation} \label{eqn:wooley}
\MV \lse P^\epsilon(P^s + P^{2s-\frac{k(k+1)}{2}}).
\end{equation}
\end{theorem}

Thus, Theorem \ref{thm:wooley} answers the question up to $\epsilon$ if $s$ is sufficiently larger and smaller than $\norm{a_k}=\frac{k(k+1)}{2}$. Let $\tilde V(k)$ be the least number $s\geq \norm{a_k}$ for which (\ref{eqn:wooley}) holds. A crude standard estimate (see Appendix \ref{sec:JaK}) gives
\begin{equation} \label{eqn:crude}
\JaC \lse P^{2s-\norm{a}+\epsilon}
\end{equation}
for $s \geq V(a_d)$. However, we expect that (\ref{eqn:crude}) holds for a larger range of $s$ in view of the lower bound (\ref{eqn:lower}). We denote by $V(a)$ the least number $s$ for which (\ref{eqn:crude}) holds. With $r=2s=2V(a)$ and $\delta = \norm{a}$, Theorem \ref{thm:var} implies the following:
\begin{thm} \label{thm:JaK}
Let $a_d \geq 3$ and $\lambda_a=1-\frac{\norm{a}}{2V(a)}$. Then $\JaK$ extends to a bounded operator from $l^p(\zn{d})$ to $l^q(\zn{d})$ if $\lambda_a < \lambda <1$ and $p,q$ satisfy 
\begin{itemize}
\item[(i)] $\frac{1}{q}  <  \frac{1}{p} - \frac{1}{\norm{a}}(1-\lambda)$ and
\item[(ii)] $\frac{1}{p} > \frac{V(a)}{\norm{a}}(1-\lambda), \frac{1}{q}  <  1-\frac{V(a)}{\norm{a}} (1-\lambda).$
\end{itemize}
\end{thm}

If one has $V(a)=\norm{a}$, then Theorem \ref{thm:JaK} would imply the nearly sharp result toward Conjecture \ref{conj:JK} for $\frac{1}{2}<\lambda < 1$. In order to extend this to the full range of $0<\lambda<1$, we need the stronger estimate
\begin{equation} \label{eqn:strongestimate}
 N^{\gamma_a}_{2\norm{a}-1}(P;h) \ll P^{\norm{a}-1+\epsilon}
\end{equation}
uniformly in $h\in \zn{d}$. Indeed, Oberlin's result for $J_{2,\lambda}$ is based on the estimate (\ref{eqn:strongestimate}) which is valid for $a=(1,2)$. See \cite{O} for details.

We record the results for the special case $\JK = {J^{a_k}_{\lambda}}$, where $a_k=(1,2,\cdots,k)$ from Theorem \ref{thm:wooley}. Theorem \ref{thm:var} with $r=2s=2(k^2-1)$ and $\delta = \frac{k(k+1)}{2}$ gives
\begin{cor} \label{thm:JK}
Let $k \geq 3$ and $\lambda_k=1-\frac{k}{4(k-1)}$. Then $\JK$ extends to a bounded operator from $l^p(\zn{k})$ to $l^q(\zn{k})$ if $\lambda_k < \lambda< 1$ and $p,q$ satisfy 
\begin{itemize}
\item[(i)] $\frac{1}{q}  <  \frac{1}{p} - \frac{2}{k(k+1)}(1-\lambda)$ and
\item[(ii)] $\frac{1}{p} > \frac{2(k-1)}{k}(1-\lambda), \frac{1}{q}  <  1-\frac{2(k-1)}{k}(1-\lambda).$
\end{itemize}
\end{cor}

Theorem \ref{thm:var} with $r=2s=2\delta = 2\lfloor\frac{(k+1)^2}{4} \rfloor$ implies
\begin{cor} \label{thm:JK2}
Let $k \geq 3$ and $\frac{1}{2}<\lambda< 1$. Then $\JK$ extends to a bounded operator from $l^p(\zn{k})$ to $l^q(\zn{k})$ if $p,q$ satisfy 
\begin{itemize}
\item[(i)] $\frac{1}{q}  <  \frac{1}{p} - \frac{1}{\lfloor\frac{(k+1)^2}{4}\rfloor}(1-\lambda)$ and
\item[(ii)] $\frac{1}{p} > 1-\lambda, \frac{1}{q}  <  \lambda.$
\end{itemize}
\end{cor}
Corollary \ref{thm:JK2} complements Corollary \ref{thm:JK} in the sense that it allows a wider range of applicable $\lambda$ and that the condition (ii) is optimal at the expense of relaxing condition (i). 

\begin{figure}[t]
\begin{tikzpicture}[scale=5][domain=0:1] 
    \draw[step=.2cm, style=help lines] (0,0) rectangle (1,1);
    \draw (0,0) -- (1,0) node[right] {$1/p$};
   \draw (0,0) -- (0,1) node[above] {$1/q$};

\foreach \x/\xtext in {0,0.3333 /\frac{1}{3}, 0.4444 /\frac{4}{9} , 0.75 /\frac{3}{4}, 1}
\draw (\x cm,0.5pt) -- (\x cm,-0.5pt) node[anchor=north] {$\xtext$};
\foreach \y/\ytext in { 0.25 /\frac{1}{4}, 0.3889 /\frac{7}{18}, 0.6667 / \frac{2}{3}, 1}
\draw (0.5pt,\y cm) -- (-0.5pt,\y cm) node[anchor=east] {$\ytext$};

\draw[dashed] (0.3333, 0) -- (0.3333, 0.2778); 
\draw (0.3333, 0.2778) -- (0.7222, 0.6667);
\draw[dashed] (0.7222, 0.6667) -- (1, 0.6667);

\draw [fill=lightgray] (0.3333, 0) -- (0.3333, 0.25) -- (0.4444, 0.3889) -- (0.6111, 0.5556) -- (0.75, 0.6667) -- (1, 0.6667) -- (1,0) -- (0.3333, 0);
\draw [white] (0.3333, 0) -- (0.3333, 0.25) -- (0.4444, 0.3889) -- (0.6111, 0.5556) -- (0.75, 0.6667) -- (1, 0.6667);
\draw [dashed] (0.3333, 0) -- (0.3333, 0.25) -- (0.4444, 0.3889) -- (0.6111, 0.5556) -- (0.75, 0.6667) -- (1, 0.6667) ;

\draw (0.3333, 0.25) circle (0.2pt); 
\draw (0.4444, 0.3889) circle (0.2pt);
\draw (0.6111, 0.5556) circle (0.2pt);
\draw (0.75, 0.6667) circle (0.2pt);
\end{tikzpicture}
\caption{\label{fig:pic} The known region of boundedness of $J_{3,\frac{2}{3}}$}
\end{figure}
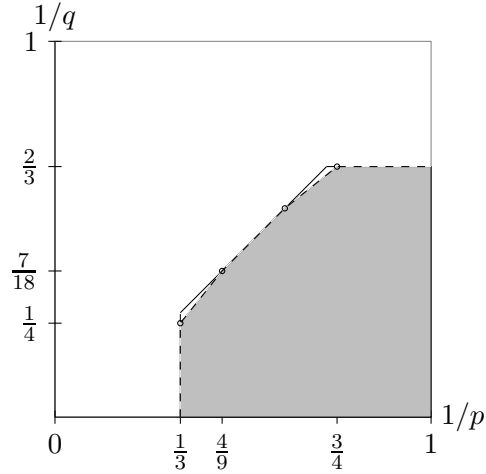

Figure \ref{fig:pic} illustrates Corollary \ref{thm:JK} and \ref{thm:JK2} for the case $k=3$ and $\lambda = \frac{2}{3}$. It shows the known (shaded) and the conjectured range of exponents $(\frac{1}{p},\frac{1}{q}) \in [0,1]^2$ where $J_{3,\frac{2}{3}}$ is bounded from $l^p(\zn{3})$ to $l^q(\zn{3})$. Corollary \ref{thm:JK} and \ref{thm:JK2} give the vertices $(\frac{4}{9},\frac{7}{18})$ and $(\frac{1}{3},\frac{1}{4})$, respectively. 

\subsection{Relation to Waring's problem} \label{sec:Waring}
The purpose of this section is twofold; to give a slight improvement to one of the results on the operator $\IK$ in \cite{P2} and to give a correction to an interpolation lemma in \cite{P2}. We thank Pierce for encouraging us to clarify the issue here.

We denote by $r_{s,k}(l)$ the number of representations of $l$ as a sum of $s$ positive $k$-th powers. Then there is an estimate
\begin{equation} \label{eqn:asy}
r_{s,k}(l) \ll l^{s/k-1+\epsilon}
\end{equation}
for every sufficiently large $s$ with respect to $k$, see \cite{V}. Let $\tilde G(k)$ to be the least natural number $s$ for which the estimate $(\ref{eqn:asy})$ holds. Since the work by Hardy and Littlewood that $\tilde G(k) \leq (k-2)2^{k-1} +5$, numerous improvements have been achieved. We refer the reader to \cite{FW} and references therein. 

In \cite{P2}, (\ref{eqn:asy}) was applied to obtain estimates on the mean values of the Weyl-sums. Instead, we use it to obtain estimates on $N^\gamma_{2s-1}(P;h)$, which is the number of the solutions of the Diophantine equation of $2s-1$ variables
\begin{equation}\label{eqn:dodd}
 x_1^k + \cdots + x_s^k = x_{s+1}^k + \cdots + x_{2s-1}^k + h
\end{equation}
for $x_i \in [1,P]\cap \mathbb{Z}$ and $h\in \mathbb{Z}$. Indeed, one has
\begin{equation} \label{eqn:essodd}
N^\gamma_{2s-1}(P;h) \ll P^{2s-1 - k + \epsilon}
\end{equation}
for $s\geq \tilde G(k)$ uniformly in $h\in \Z$.

For the convenience of the reader, we record the argument here. One first considers the number of solutions $(x_1,\cdots,x_s)$ for each fixed $(x_{s+1},\cdots,x_{2s-1})$. It is $O(P^{s-k+\epsilon})$ uniformly in $h$ by (\ref{eqn:asy}) since we may assume that the right hand side of (\ref{eqn:dodd}) is $O(P^k)$. We get (\ref{eqn:essodd}) since there are $P^{s-1}$ many choices for $(x_{s+1},\cdots,x_{2s-1})$.

The estimate (\ref{eqn:essodd}) and Theorem \ref{thm:var} with $r=2\tilde G(k)-1$ give the following which slightly improves the allowable range of $\lambda$ in \cite{P2}.
\begin{thm} \label{thm:IK}
Let $\lambda_{k} = 1-\frac{k}{2\tilde G(k)-1}$. Then
$\IK$ extends to a bounded operator from $l^p(\mathbb{Z})$ to $l^q(\mathbb{Z})$ if $\lambda_{k} < \lambda < 1$ and $p,q$ satisfy
\begin{itemize}
\item[(i)] $\frac{1}{q}  <  \frac{1}{p} - \frac{1}{k}(1-\lambda)$ and
\item[(ii)] $\frac{1}{p} > \frac{\tilde G(k)}{k}(1-\lambda), \frac{1}{q}  <  1-\frac{\tilde G(k)}{k} (1-\lambda)$.
\end{itemize}
\end{thm}

We note that the optimal condition (ii-c) $1/q < \lambda, 1/p > 1-\lambda$ in the statement of \cite[Theorem 4]{P2} should be replaced by the weaker condition (ii) in Theorem \ref{thm:IK}. This is due to an error in the interpolation lemma \cite[Lemma 2]{P2}. The condition (ii) $1/q < \lambda, 1/p > 1-\lambda$ in the lemma should be corrected to a weaker condition (ii) $1/p > \frac{1-\lambda}{2(1-\eta)}, 1/q < 1-\frac{1-\lambda}{2(1-\eta)}$.

\section{Proof of Theorem \ref{thm:var}}   \label{sec:pf}
The theorem can be reduced to certain restricted weak type estimates. We decompose the operator $\I$ dyadically.
Define $\II$ by
\begin{equation*}
\II(f)(n) = 2^{-\lambda j} \sum_m f(n-\gamma(m)) 
\end{equation*}
for $j\geq 0$ where $\sum_m$ denotes $\sum_{2^j \leq m < 2^{j+1} }$. 

Oberlin \cite{O} used the following lemma contained in the proof of Lemma $1$ of \cite{C}.
\begin{lemma}[Christ \cite{C}] \label{lem:chr} Suppose that $T$ is an operator taking characteristic functions $\chi_E$ onto measurable functions with $T\chi_E \geq 0$ for any measurable set $E$. Given $\alpha>0$ and $E$ with $0<|E|<\infty$, take $F = \{x: \alpha < T\chi_E(x) < 2\alpha \}$ and $\beta = \abs{E}^{-1}\inn{\chi_F, T\chi_E}$. For $k = 0, 1, \dots$ there are positive constants $\delta_k$ and $\epsilon_k$ (depend only on $k$) such that the sets $E_k$ and $F_k$ defined by $E_0 = E$, $F_0 = F$,
\begin{align*}
E_{k+1} &= \{x \in E_k : T^*\chi_{F_k}(x) \geq \delta_k \beta \}, \\
F_{k+1} &= \{y \in F_k : T\chi_{E_{k+1}}(y) \geq \epsilon_k\alpha\},
\end{align*}
are nonempty provided that $|F|>0$.
\end{lemma}

Let $T=\II, \alpha>0, E\subset \zn{d}$ and $F\equiv F^j, \beta \equiv {\beta}_j, E_k\equiv E^j_k, F_k\equiv F^j_k$ be as in Lemma \ref{lem:chr}. We may assume that $|F|>0$.

Note that $n \in E_k$ implies
\begin{equation} \label{eqn:beta2}
\sum_m \chi_{F_{k-1}}( n + \gamma(m) ) \geq 2^{\lambda j} \delta_{k-1}\beta
\end{equation}
and $n \in F_k$ implies
\begin{equation}  \label{eqn:alpha2}
\sum_m \chi_{E_{k}}( n - \gamma(m) ) \geq 2^{\lambda j} \epsilon_{k-1}\alpha
\end{equation}
for $k\geq 1$. 

When $r= 2s -1$, we define the sum $S_{r}$ by
\begin{equation*}
S_{r}  =  \sum_{m_1} \sum_{m_2} \cdots \sum_{m_{r}} \chi_E \left( n + \sum_{i=1}^{r} (-1)^i \gamma(m_i) \right)
\end{equation*} 
for a fixed $n \in F_{s-1}$.

Since $n \in F_{s-1}$, there are at least $2^{\lambda j} \epsilon_{s-2}\alpha$ values of $m_1 \in [2^j,2^{j+1})$ such that $n-\gamma(m_1)  \in E_{s-1}$ by ($\ref{eqn:alpha2}$). For each of these $m_1$, there are at least $ 2^{\lambda j} \delta_{s-2}\beta$ values of $m_2 \in [2^j,2^{j+1})$ such that $n-\gamma(m_1) + \gamma(m_2) \in F_{s-2}$ by ($\ref{eqn:beta2}$). Continuing in this manner, we get the following lower bound of $S_{r}$:
\begin{equation} \label{eqn:comb}
S_{r} \gtrsim 2^{(2s-1)\lambda j} \alpha^{s} \beta^{s-1} \gtrsim 2^{r\lambda j} \alpha^{r}\frac{|F|^{s-1}}{|E|^{s-1}}
\end{equation}
since $\beta \geq \alpha \frac{|F|}{|E|}$ and $r=2s-1$.

Next, we get an upper bound for $S_r$. Let $E' = n - E$.
\begin{equation} \label{eqn:fi}
\begin{split}
S_{r}  &=  \sum_{m_1} \sum_{m_2} \cdots \sum_{m_{r}} \chi_{E'} \left( \sum_{i=1}^{r} (-1)^{i+1} \gamma(m_i) \right) \\
 &= \sum_{l \in E'}\sumab{\sum_{i=1}^{r} (-1)^{i+1} \gamma(m_{i}) = l}{2^j \leq m_i < 2^{j+1}} 1 \leq \sum_{l \in E'} N^\gamma_{r}(2^{j+1};l) \lse 2^{(r-\delta+\epsilon) j} |E|,
\end{split}
\end{equation}
where we fix $\epsilon>0$ so that $\lambda = \frac{r- \delta + 2\epsilon}{r}$.
$(\ref{eqn:comb})$ and $(\ref{eqn:fi})$ give
\begin{equation*} 
\alpha^{r} |F|^{s-1} \lse 2^{-\epsilon j}  |E|^{s},
\end{equation*}
or equivalently,
\begin{equation} \label{eqn:final}
\alpha^{r} |\{\II(\chi_E)(n) >\alpha \}|^{s-1}  \lse 2^{-\epsilon j}  |E|^{s}.
\end{equation}

Since $\I(\chi_E)(n) \leq \sum_{j=0}^\infty \II(\chi_E)(n)$, $(\ref{eqn:final})$ implies 
\begin{equation*} 
\alpha^{r} |\{\I(\chi_E)(n) >\alpha \}|^{s-1} \lse |E|^{s},
\end{equation*}
which is equivalent to the restricted weak-type $(\frac{2s-1}{s},\frac{2s-1}{s-1})$ estimate for $\I$. 

When $r=2s$, we take the sum $S_{r}$ by
\begin{equation*}
S_{r}  =  \sum_{m_1} \sum_{m_2} \cdots \sum_{m_{r}} \chi_E \left( n + \sum_{i=1}^{r} (-1)^{i+1} \gamma(m_i) \right)
\end{equation*}
for a fixed $n \in E_{s}$. After a few similar estimates, we get the restricted weak-type $(\frac{2s}{s+1},2)$ estimate for $\I$.

Moreover, the same estimates are valid for a complex-valued $\lambda$ as long as $\Re(\lambda) > 1-\frac{\delta}{r}$. For each fixed such $\lambda$, we first obtain strong type estimates with bounds uniform in $\Im(\lambda)$ by real interpolation with the $l^1 \to l^\infty$ bound for $\Re(\lambda) \geq 0$. Next, we apply analytic interpolation with the $l^\infty \to l^\infty$ bound for $\Re(\lambda) > 1$. Finally, inclusion property of $l^p$ spaces completes the proof. We refer the reader to \cite[Chapter V]{SWe} for the interpolation theorems used here.

\section{Discrete fractional integral along the hyperbolic paraboloid in $\zn{3}$} \label{sec:hyp}
The discrete fractional integrals, where the summation is taken along positive definite quadratic forms in several variables, have been studied by Pierce \cite{P1} via studying the Fourier multiplier of the operator. Motivated by \cite{P1}, we study the discrete fractional integral along the hyperbolic paraboloid in $\zn{3}$, defined by
\begin{equation*}
\mathcal{P}_\lambda f(n_1,n_2,n_3) = \sum_{m \in \zn{2}\setminus 0} \frac{f(n_1-m_1, n_2-m_2, n_3- (m_1^2-m_2^2))}{|m|^{2\lambda}},
\end{equation*}
acting on (initially) compactly supported functions $f:\zn{3} \to \mathbb{C}$. 

\begin{thm} \label{thm:JL}
Let $0<\lambda< 1$. Then $\mathcal{P}_{\lambda}$ extends to a bounded operator from $l^p(\zn{3})$ to $l^q(\zn{3})$ if $p,q$ satisfy
\begin{itemize}
\item[(i)] $\frac{1}{q}  <  \frac{1}{p} - \frac{1}{2}(1-\lambda)$ and
\item[(ii)] $\frac{1}{p} > 1-\lambda, \frac{1}{q}  <  \lambda$.
\end{itemize}
\end{thm}
This result is sharp up to endpoint. One can show the necessity of the condition $\frac{1}{q}  \leq  \frac{1}{p} - \frac{1}{2}(1-\lambda)$ by taking $f(n) = |(n_1,n_2)|^{-\alpha}|n_3|^{-\beta}$ for $n_j \geq 1$ and $f(n)=0$ otherwise, for some appropriate $\alpha,\beta>0$. The necessity of condition (ii) can be shown as in Appendix \ref{sec:app3}.

To treat the operator $\mathcal{P}_\lambda$, we consider a variant of $\I$ defined by
\begin{equation*}
\IV(f)(n) = \sum_{m\in \zn{d_0} \setminus 0} \frac{f(n- \gamma(m))}{|m|^{d_0\lambda}}
\end{equation*}
acting on (initially) compactly supported functions $f:\zn{d} \to \mathbb{C}$, where $\gamma : \zn{d_0} \to \zn{d}$ is an injection. 

For fixed $h \in \zn{d}$, $r,P \in \mathbb{N}$, let $\AN$ denote be the number of solutions of the Diophantine system 
\begin{equation*}
\sum_{i=1}^r (-1)^{i+1} \gamma(m_i) = h
\end{equation*}
for $m_i \in B_P$, where $B_P=\{ x\in \zn{d_0}:|x|\leq P\}$ and $|x|^2=x_1^2+\cdots +x_{d_0}^2$. 

We have the following variant of Theorem \ref{thm:var}.
\begin{thm} \label{thm:var2}
Suppose that for each $\epsilon>0$ we have $\AN \lse P^{d_0(r-\delta) + \epsilon}$ for a fixed $r\in\mathbb{N}$ and $\delta>0$ uniformly in $h \in \zn{d}$. Let $s \in\mathbb{N}$ be the number such that we have either $r=2s$ or $r=2s-1$. Then $\IV$ extends to a bounded operator from $l^p(\zn{d})$ to $l^q(\zn{d})$ if $1-\frac{\delta}{r}<\lambda< 1$ and $p,q$ satisfy 
\begin{itemize}
\item[(i)] $\frac{1}{q} < \frac{1}{p} - \frac{1-\lambda}{\delta}$ and
\item[(ii)] $\frac{1}{p}> \frac{s}{\delta}(1-\lambda)$, $\frac{1}{q} < 1 - \frac{s}{\delta}(1-\lambda)$.
\end{itemize}
\end{thm} 
The proof of Theorem \ref{thm:var2} is a straightforward modification of the proof of Theorem \ref{thm:var} and will be omitted.

\begin{proof} [Proof of Theorem \ref{thm:JL}]
By Theorem \ref{thm:var2}, it is enough to show that
\begin{equation} \label{eqn:goal}
\mathcal{N}_3^\gamma (P;h) \lse P^{2+\epsilon}
\end{equation}
uniformly in $h\in \zn{3}$.

Here, $d_0 = 2$, $m=(m_1,m_2) \in \zn{2}$ and $\gamma(m) = (m, \tau(m))$, where $\tau(m) = m_1^2 - m_2^2$. Recall that $\mathcal{N}_3^\gamma (P;h)$ is the number of solutions of the Diophantine system
\begin{equation*}
\begin{split}
x + y &= z + v \\
\tau(x) + \tau(y) &= \tau(z) + t
\end{split}
\end{equation*}
for $x,y,z \in B_P$ and $h=(v,t)\in \zn{2+1}$. 

Since there are $O(P^2)$ many $z$ in $B_P$, it is enough to show that the number of solutions of 
\begin{equation} \label{eqn:bd}
\begin{split}
x + y &= v' \\
\tau(x) + \tau(y) &= t'
\end{split}
\end{equation}
for $x,y \in B_P$ is $O(P^\epsilon)$ uniformly in $h'=(v',t')\in \zn{2+1}$.

Suppose that $(x,y)$ is a solution of (\ref{eqn:bd}). Then $|v'|\leq 2P$ and $|t'| \leq 2P^2$. We make a change of variables $X = 2x - v'$ and $Y = 2y - v'$. Then $(X,Y)$ is a solution of the Diophantine system
\begin{equation*} 
\begin{split}
X+Y &= 0\\
\tau(v' + X) + \tau(v' + Y) &= 4t',
\end{split}
\end{equation*}
which is equivalent to the Diophantine equation
\begin{equation} \label{eqn:bd3}
 (X_1 + X_2)(X_1 - X_2) = N := 2t'-({v'_1}^2 - {v'_2}^2),
\end{equation} 
where $X = (X_1,X_2)$ and $v'=(v'_1,v'_2)$. 

Therefore, the number of solutions $(X_1,X_2)$ of (\ref{eqn:bd3}) is $O(d(N))$, where $d(N)$ is the number of divisors of $N$. The fact that $d(N) = O( |N|^\epsilon)$ (see Chapter 18 of \cite{HW}) and $|N| =O(P^2)$ implies that the number of solutions $(X_1,X_2)$ of (\ref {eqn:bd3}) is $O(P^\epsilon)$ for any $\epsilon>0$, which in turn implies that the number of solutions of (\ref{eqn:bd}) is $O(P^\epsilon)$.

Theorem \ref{thm:var2} with $r=3$ and $\delta = 2$ implies Theorem \ref{thm:JL} for $\frac{1}{3} < \lambda <1$. Interpolating the result with the trivial $l^1(\zn{3}) \to l^\infty(\zn{3})$ bound for $\Re(\lambda) \geq 0$ finishes the proof.
\end{proof}

\section{Appendix}
\subsection{Necessity of conditions in Conjecture \ref{conj:JK}.} \label{sec:app3}
For the necessity of the second condition, we take the example in \cite{SW}; $f(0)=1$, $f(n)=0$ for $n\neq 0$. Then $f \in l^p (\zn{d})$ for all $p$, and
\begin{displaymath}
\JaK (f)(n) = \left\{ \begin{array}{ll}
m^{-\lambda} & \textrm{if $n=\gamma^a(m)$ for some $m\in \mathbb{N}$}\\
0 & \textrm{otherwise.}
\end{array}
 \right.
\end{displaymath}
Thus
\begin{equation*}
\norm{\JaK(f)}_{l^q( \zn{d} )}^q = \sum_{m\geq 1} m^{-\lambda q},
\end{equation*}
where the sum converges only if $1/q< \lambda$. Duality gives $1/p > 1-\lambda$.

For the necessity of the first condition, we take $f:\zn{d} \to \mathbb{C}$ by
\begin{displaymath}
f(n) = \left\{ \begin{array}{ll}
\prod_{j=1}^d |n_j|^{-\alpha} & \textrm{if $n_j \neq 0$ for all $1 \leq j \leq d$}\\
0 & \textrm{otherwise}
\end{array}
 \right.
\end{displaymath}
for a fixed constant $\alpha > 1/p$ so that  $f \in l^p (\zn{d})$.

For $n_1>1$ and $n_j \geq n_1^{a_j/a_1}$ for $2\leq j\leq d$, 
\begin{align*}
\JaK (f)(n) &\geq \sum_{1\leq m^{a_1} < n_1} m^{-\lambda} \prod_{j=1}^d (n_j-m^{a_j})^{-\alpha}  \gtrsim n_1^{(1-\lambda)/a_1} \prod_{j=1}^d n_j^{-\alpha}.
\end{align*}
Thus,
\begin{align*}
\norm{\JaK(f)}_{l^q( \zn{d} )}^q &\gtrsim  \sum_{n_1>1} n_1^{q(1-\lambda)/a_1 -q\alpha}  \sum_{\substack{n_1^{a_j} \leq n_j^{a_1} \leq 2n_1^{a_j} \\ 2\leq j\leq d}}  \prod_{j=2}^d n_j^{-q\alpha}\\ 
&\gtrsim \sum_{n_1>1}   n_1^{q(1-\lambda)/a_1 -q\alpha} \prod_{j=2}^d n_1^{a_j(1-q\alpha)/a_1} \\
&\gtrsim \sum_{n_1>1}  n_1^{q(1-\lambda)/a_1 + \norm{a}(1-q\alpha)/a_1-1} ,
\end{align*}
where the last sum converges only if $\frac{1}{q}  < \alpha - \frac{1}{\norm{a}}(1-\lambda)$. We get the first necessary condition since we may decrease $\alpha$ to $1/p$ as close as we want.

\subsection{Proof of (\ref{eqn:crude})} \label{sec:JaK}
Let $a=(a_1,\cdots,a_d)$ be given. Let $l=a_d-d$ and $\{b_i\}_{i=1}^l$ be the increasing sequence of natural numbers such that $\{ b_1,\cdots,b_l \} = ([1,a_d]\cap \mathbb{Z}) \setminus \{ a_1,\cdots,a_d\}$. Considering the underlying Diophantine systems, we have
\begin{equation} \label{eqn:VJ}
\begin{split}
\JaC &= \sum_{|h_{b_1}| \leq sP^{b_1}} \cdots \sum_{\abs{h_{b_{l}}} \leq sP^{b_l}} \int_{[0,1]^{a_d}} |S(\alpha)|^{2s} \prod_{i=1}^{l} e(-h_{b_i} \alpha_{b_i}) d\alpha \\
&\ll P^{\sum_{i=1}^{l} b_i} \mathbf{J}_{s,a_d}(P) =  P^{\frac{a_d(a_d+1)}{2}-\norm{a}} \mathbf{J}_{s,a_d}(P),
\end{split}
\end{equation}
where $S(\alpha) = \sum_{m=1}^P e(\alpha_1 m + \alpha_2 m^2+ \cdots + \alpha_{a_d} m^{a_d})$.

Combining (\ref{eqn:wooley}) and (\ref{eqn:VJ}), we have
\begin{equation} \label{eqn:Vino}
\JaC \lse P^{2s-\norm{a}+\epsilon}
\end{equation}
for $s\geq \tilde V(a_d)$.

\subsection{Fourier multiplier approach for $\JaK$} \label{sec:alt}
Let $\hat{f}(\alpha) = \sum_{n\in \zn{d}} f(n) e(-n\cdot \alpha)$ be the Fourier transform of $f \in l^1(\zn{d})$, where $e(t)\equiv e^{2\pi i t}$. We shall study the Fourier multiplier $m^a_\lambda$ of the operator $\JaK$ 
\begin{equation*}
m^a_\lambda(\alpha) = \sum_{n=1}^\infty \frac{e({-\gamma^a(n) \cdot \alpha})}{n^\lambda}
\end{equation*}
for $\alpha \in [0,1]^d$, given by the relation $\widehat{\JaK (f)}(\alpha) = m^a_{\lambda}(\alpha) \hat{f}(\alpha)$.

By closely following the argument in \cite{SW,P2}, we generalize the Weyl sum approach on the Fourier multipliers (Proposition 5 of \cite{P2}) as follows:
\begin{prop} \label{prop:weyl} Let $s\in \mathbb{N}$, $0<\delta\leq 2s$, and $\lambda\in \mathbb{C}$. $\JaC \lse P^{2s-\delta + \epsilon}$ for each $\epsilon>0$ if and only if $m^a_\lambda \in L^{2s}([0,1]^d)$ for all $\Re(\lambda)> 1-\frac{\delta}{2s}$.
\end{prop}

The ``folk" lemma (Lemma 2 of \cite{SW}) implies the following result. 
\begin{cor} \label{cor:Pi} Suppose that one has $\JaC \lse P^{2s-\delta + \epsilon}$ for some $0<\delta \leq 2s$. Then the operator $\JaK$ extends to a bounded operator from $l^{\frac{2s}{s+1}}(\zn{d})$ to $l^2(\zn{d})$ for $\Re(\lambda)> 1-\frac{\delta}{2s}$.
\end{cor}
Note that Theorem \ref{thm:JaK} follows from Corollary \ref{cor:Pi} and a complex interpolation.

\begin{proof}[Proof of Proposition \ref{prop:weyl}]
For $l\in \zn{d}$, let $r^a_s(l)$ be the number of solutions of the Diophantine system
\begin{equation} \label{eqn:en1} x_1^{a_j} + \cdots + x_s^{a_j} = l_j,
\end{equation}
where $l=(l',l_d)=(l_1, \cdots, l_d)$ and $x_i \geq 1$ for $1\leq j \leq d$. In addition, we define $r^a_s(l;P)$ to be the number of solutions of (\ref{eqn:en1}) for $1 \leq x_i \leq P$.

Let $\delta>0$ be given. We observe that $\JaC \lse P^{2s-\delta+\epsilon}$ is equivalent to 
\begin{equation} \label{eqn:en2}
\sum_{l_d=1}^P R^a_s(l_d) \lse P^{(2s-\delta+\epsilon)/a_d}, \text{ where } R^a_s(l_d) = \sum_{l' \in \zn{d-1}} \left( r^a_s (l',l_d)\right)^2.
\end{equation}
This follows from Parseval's identity applied to $\JaC = \int_{[0,1]^d} |S^{\gamma^a}(\alpha)|^{2s} d\alpha$, 
where $S^{\gamma^a}(\alpha) = \sum_{m=1}^{P} e(\gamma^a(m)\cdot \alpha).$
Indeed, one has
\begin{align*}
\JaC = \sum_{l\in \zn{d}} (r^a_s(l;P))^2 = \sum_{1\leq l_d \leq sP^{a_d}} \sum_{l'\in \zn{d-1}} (r^a_s(l',l_d))^2 = \sum_{1\leq l_d \leq sP^{a_d}} R^a_s(l_d).
\end{align*}

Therefore, it is enough to show that (\ref{eqn:en2}) is equivalent to 
\begin{equation} \label{eqn:mul}
m^a_\lambda \in L^{2s}([0,1]^{d}),  \text{ for every } \Re(\lambda)> 1-\frac{\delta}{2s}.
\end{equation}

As in \cite{SW,P2}, one has 
\begin{equation*}
m^a_\lambda (\alpha) = C_{a,\lambda} \int_0^1 S^a_y(\alpha) y^{\lambda/a_d}\frac{dy}{y} + O(1),
\end{equation*} 
where $C_{a,\lambda} = \Gamma(\lambda/a_d)$ and
$S^a_y(\alpha) = \sum_{n\geq 1} e^{-n^{a_d}y} e(-\gamma^a(n)\cdot \alpha)$ which is well-defined for each $y>0$.
This follows from the observation 
\begin{equation*}
\int_0^\infty e^{-n^{a_d}y} y^{\lambda/a_d}\frac{dy}{y} = n^{-\lambda} \Gamma(\lambda/a_d).
\end{equation*}

Thus, 
\begin{equation} \label{eqn:nmul}
\norm{m^a_\lambda}_{L^{2s}([0,1]^{d})} \leq C_{a,\lambda} \int_0^1 \norm{S^a_y}_{L^{2s}([0,1]^{d})} y^{\Re(\lambda)/a_d} \frac{dy}{y} + O(1).
\end{equation}
By Parseval's identity and summation by parts, (\ref{eqn:en2}) implies
\begin{equation} \label{eqn:nmul2}
 \norm{S^a_y}_{L^{2s}([0,1]^{d})}^{2s} = \sum_{l_d \geq 1} e^{-2l_dy} R^a_s(l_d) \lse y^{-(2s-\delta+\epsilon)/a_d}.
\end{equation}
Thus we get (\ref{eqn:mul}) by (\ref{eqn:nmul}) and (\ref{eqn:nmul2}).

For the converse, it is enough to assume (\ref{eqn:mul}) only for $\lambda\in \mathbb{R}$. Then $m^a_\lambda(\alpha)^s = \sum_{l\in \zn{d}} a_l e(-l\cdot \alpha) \in L^2$,
where 
\begin{equation*}
a_l = \sum_{n_1, \cdots, n_s \geq 1} \sum_{\gamma^a(n_1)+\cdots+\gamma^a(n_s)=l} \frac{1}{n_1^\lambda \cdots n_s^\lambda}.
\end{equation*}
Parseval's identity implies that $\sum_{l\in \zn{d}} |a_l|^2$ is finite. Since $n_i^{a_d} \leq l_d$ in the sum, we have $a_l \geq r^a_s(l) l_d^{-s\lambda/a_d}$ if $l_d \geq 1$. Thus,

\begin{align*}
\infty > \sum_{l\in \zn{d}} |a_l|^2 \geq \sum_{l_d\geq 1} \sum_{l'\in \zn{d-1}} (r^a_s(l))^2 l_d^{-2s\lambda/a_d} =\sum_{l_d\geq 1} R^a_s(l_d) l_d^{-2s\lambda/a_d}
\end{align*}
for every $\lambda> 1-\frac{\delta}{2s}$, which is equivalent to (\ref{eqn:en2}).
\end{proof}

\section*{Acknowledgments}
The author is grateful to a referee for the valuable comments, which improved the results of the paper. The author would like to thank Jong-Guk Bak for suggesting this research topic and his guidance, and Lillian Pierce for her encouragement. The author also would like to thank Jungho Park for helpful discussion on the circle method and Stephen Wainger for pointing out a misquotation. This paper grew out of the author's Master's thesis \cite{K}, and Theorem \ref{thm:JL} is contained in the thesis.

\begin{thebibliography}{20}

\bibitem{B}
J.~Bourgain, \emph{Estimations de certaines fonctions maximales}, C. R. Acad.
  Sci. Paris S\'er. I Math. \textbf{301} (1985), no.~10,  499--502.

\bibitem{CSWW}
A.~Carbery, A.~Seeger, S.~Wainger, and J.~Wright, \emph{Classes of singular
  integral operators along variable lines}, J. Geom. Anal. \textbf{9} (1999),
  no.~4,  583--605.

\bibitem{C1}
M.~Christ, \emph{Endpoint bounds for singular fractional integral operators}.
  Unpublished manuscript, 1988.

\bibitem{C}
---{}---{}---, \emph{Convolution, curvature, and combinatorics: a case study},
  Internat. Math. Res. Notices  (1998), no.~19,  1033--1048.

\bibitem{D}
H.~Davenport, Analytic methods for {D}iophantine equations and {D}iophantine
  inequalities, Cambridge Mathematical Library, Cambridge University Press,
  Cambridge, second edition (2005).

\bibitem{KF}
K.~B. Ford, \emph{New estimates for mean values of {W}eyl sums}, Internat.
  Math. Res. Notices  (1995), no.~3,  155--171.

\bibitem{FW}
K.~B. Ford and T.~D. Wooley, \emph{On Vinogradov's mean value theorem: Strongly
  diagonal behaviour via efficient congruencing}, Acta Math \hskip -.1cm. (to
  appear).

\bibitem{HW}
G.~H. Hardy and E.~M. Wright, An introduction to the theory of numbers, Oxford
  University Press, Oxford, sixth edition (2008).

\bibitem{IW}
A.~D. Ionescu and S.~Wainger, \emph{{$L\sp p$} boundedness of discrete singular
  {R}adon transforms}, J. Amer. Math. Soc. \textbf{19} (2006), no.~2,
  357--383.

\bibitem{K}
J.~Kim, On discrete fractional integral operators and related Diophantine
  equations, Master's thesis, Pohang University of Science and Technology
  (2011).

\bibitem{O}
D.~M. Oberlin, \emph{Two discrete fractional integrals}, Math. Res. Lett.
  \textbf{8} (2001), no. 1-2,  1--6.

\bibitem{Pt}
L.~B. Pierce, Discrete analogues in harmonic analysis (2009). Thesis
  (Ph.D.)--Princeton University.

\bibitem{P2}
---{}---{}---, \emph{On discrete fractional integral operators and mean values
  of {W}eyl sums}, Bull. Lond. Math. Soc. \textbf{43} (2011), no.~3,  597--612.

\bibitem{P1}
---{}---{}---, \emph{Discrete fractional {R}adon transforms and quadratic
  forms}, Duke Math. J. \textbf{161} (2012), no.~1,  69--106.

\bibitem{SW}
E.~M. Stein and S.~Wainger, \emph{Discrete analogues in harmonic analysis.
  {II}. {F}ractional integration}, J. Anal. Math. \textbf{80} (2000) 335--355.

\bibitem{SW2}
---{}---{}---, \emph{Two discrete fractional integral operators revisited}, J.
  Anal. Math. \textbf{87} (2002) 451--479.

\bibitem{SWe}
E.~M. Stein and G.~Weiss, Introduction to {F}ourier analysis on {E}uclidean
  spaces, Princeton University Press, Princeton, N.J. (1971).

\bibitem{V}
R.~C. Vaughan, The {H}ardy-{L}ittlewood method, Vol. 125 of \emph{Cambridge
  Tracts in Mathematics}, Cambridge University Press, Cambridge, second edition
  (1997).

\bibitem{W}
T.~D. Wooley, \emph{Vinogradov's mean value theorem via efficient
  congruencing}, Ann. of Math. (2) \textbf{175} (2012), no.~3,  1575--1627.

\bibitem{W2}
---{}---{}---, \emph{Vinogradov's mean value theorem via efficient
  congruencing, {II}}, Duke Math. J. \textbf{162} (2013), no.~4,  673--730.

\end{thebibliography}

\end{document}